\documentclass[a4paper, 12pt]{amsart}

\usepackage{amssymb}
\usepackage{amsmath}
\usepackage[T1]{fontenc}
\usepackage[utf8]{inputenc}
\usepackage[top=3cm, bottom=3cm, left=3cm, right=3cm]{geometry}
\usepackage[square, numbers]{natbib} 
\usepackage{abstract}

\usepackage{mathptmx}

\usepackage{amsthm}

\newtheorem{theorem}{Theorem}
\newtheorem{lemma}{Lemma}
\newtheorem{proposition}{Proposition}
\newtheorem{corollary}{Corollary}

\renewcommand{\phi}{\varphi}

\makeatletter
\renewcommand{\subsection}{\@startsection {subsection}{1}{\z@}%
             {-3.5ex \@plus -1ex \@minus -.2ex}%
             {2.3ex \@plus.2ex}%
             {\normalfont\normalsize\bfseries}}
\makeatother
\title[Sommerfeld Radiation for long-range potential]{Sommerfeld Radiation Condition for Helmholtz equations with long-range potentials }
\author{Eric Ströher}  

\begin{document}
\maketitle

\begin{abstract}
We study the electric Helmholtz equation $\Delta u + Vu + \lambda u =f$ and show that, for certain potentials, the solution $u$ given by the limited absorption principle obeys a Sommerfeld radiation condition. We use a non-spherical approach based on the solution $K$ of the eikonal equation $|\nabla K|^2=1 + \frac{p}{\lambda}$ to improve previous results in that area and extend them to long-range potentials which decay like $|x|^{-2-\alpha}$ at infinity, with $\alpha > 0$. 
\end{abstract}

\section{Introduction} 
Consider the electric Helmholtz equation 
\begin{equation}
\label{Helmholtz}
\Delta u + Vu + \lambda u =f.
\end{equation} 
In many fields of physics, such as the theory of diffusion, wave equations, or in quantum mechanics, problems can be reduced to finding the solution of such a Helmholtz equation. 
However, even if the existence can be proven, this solution does not need to be unique. For example, consider the equation describing the radiation of a point source, 
\begin{equation}
\label{example}
\Delta u + k^2 u =\delta (x).
\end{equation}
It has infinitely many solutions, for example 
\begin{align*}
u_{\pm}(x)=\frac{e^{\pm i k |x|}}{4 \pi |x|},
\end{align*}
as well as any linear combinations of these functions. The solution $u_{+}$ represents a wave originating from the point source, while the solution $u_{-}$ describes a wave coming from infinity. From a physical point of view, the second solution is meaningless, and should be rejected. In 1912, Sommerfeld \cite{sommerfeld_greensche_1912} proposed a condition eliminating all solutions except the outgoing waves. He showed that, if the solution of the Helmholtz equation  also satisfies the Sommerfeld radiation condition   
\begin{equation}
\label{Sommerfeld_Cond}
\lim_{|x| \rightarrow \infty} |x|^{\frac{d-1}{2}} \left( \frac{\partial}{\partial |x|} - ik \right) u(x)=0
\end{equation} uniformly in all directions, then that solution is unique. 

A few decades later, Rellich \cite{rellich_uber_1943} and Atkinson \cite{atkinson_lxi_1949} showed independently that the proof of uniqueness does not require the quantity $|x|u$ to be finite as $|x|$ tends to infinity, a condition required in Sommerfeld's proof. Several authors have extended the result to problems with mixed boundary conditions, see for example Wilcox \cite{wilcox_generalization_1956} or Levine \cite{levine_uniqueness_1964}. In more recent years, the Sommerfeld condition was used in a number of different contexts, from fluid-structure interactions (Xing \cite{xing_investigation_2008}) to Dirac operators (Kravchenko \cite{kravchenko_analogue_2000}). For a more complete introduction to the Sommerfeld radiation condition and its impact on mathematics, we recommend a note by Schot \cite{schot_eighty_1992}.

Instead of using \eqref{Sommerfeld_Cond} as a condition, it can also be viewed as a property of certain solutions. If we consider $V$ a potential such that $V\leq C|x|^{-2-\alpha}$ at infinity, with $\alpha > 1/6$, the solution $u$ of 
\begin{align}
\label{res_sol}
u=R(\lambda+i0)f:=\lim_{\epsilon \rightarrow \infty}R(\lambda+i\epsilon)f,
\end{align} 
known as the \textit{limited absorption principle}, satisfies a Sommerfeld radiation condition, as was proven by Eidus \cite{eidus_principle_1963} in 1962. Here  
$R(\lambda + i\epsilon):= (\Delta +V+\lambda u + i\epsilon)^{-1}$ is the resolvent operator of the modified Helmholz equation 
\begin{equation}
\Delta u + Vu + \lambda u + i \epsilon =f.
\end{equation}

In 1972, Ikebe and Saito \cite{ikebe_limiting_1972} extended this result to potentials of the form $V=p+Q$, where $p$ is a long-range potential, while $Q$ is a short range potential.
In the works mentioned so far, the radiation conditions can be called spherical, as they depend on $|x|$. In 1978, Saito \cite{saito_schrodinger_1987} proposed a way to improve previous results by using a non-spherical argument. Instead of working on a sphere defined bx $|x|=R$, he instead used the surface $|K(x)|=R$, where $K$ is the solution of the eikonal equation 
\begin{equation}
\label{Eikonal}
|\nabla K|^2=1 + \frac{p}{\lambda}.
\end{equation}
 where $p$ is a long-range potential, and $\lambda > 0$. With this strategy, he showed that a Sommerfeld radiation condition of the type
\begin{align*}
\int_{\mathbb{R}^d}|\nabla u -i (\nabla K )u|^2\frac{dx}{(1+|x|)^{1-\delta}} < \infty
\end{align*}
holds for solutions of \eqref{res_sol}.

In order to state the theorems of this paper, we need to introduce some notation. We define the Agmon-Hörmander norms of $u$ by 
\begin{equation*}
|||u|||_{R_0}:=\sup_{R\geq R_0}\left(\frac{1}{R}\int_{|x|\leq R}|u(x)|^2\right)^{1/2}
\end{equation*}
and 
\begin{align*}
N_{R_0}(u):=\sum_{j>J}\left( 2^{j+1}\int_{C(j)}|u|^2 \right)^{1/2} + \left( R_0\int_{|x| \leq R_0} |u|^2 \right)^{1/2}
\end{align*}
where $J$ is such that $2^{J-1} < R_0 < 2^J$ and $C(j)=\lbrace x \in \mathbb{R}^d : 2^j \leq |x| \leq 2^{j+1}\rbrace$. 

More recently, Barcelo, Vega and Zubeldia \cite{barcelo_forward_2012} studied the limited absorption principle and showed the following theorem:
\begin{theorem}
\label{BVZ_Theorem}
For $d \geq 3$, let $V$ be a potential such that $|V| \leq C|x|^{-3-\alpha}$ when $|x| \geq 1$, $C, \alpha > 0$. Then, for $u$ given by \eqref{res_sol}, we have
\begin{align*}
 \sup_{R\geq 1} R\int_{|x|\geq R} |\nabla (e^{-i\lambda^{1/2}|x|}u)|^2
< C\left[ (N_1(f))^2 + \int |x|^3 |f|^2 \right]
\end{align*}
\end{theorem}

In their paper, the authors hinted at the fact that one might be able to relax the conditions imposed on the decay of the potential $V$ by using the same approach as Saito and working with a non-spherical argument, using the solution of the eikonal equation \eqref{Eikonal}.

\pagebreak
In this paper, we confirm that such an approach indeed works, allowing us to consider the following class of potentials: let $V=p+Q$ be an electric potential, with $p\in C^{\infty}$, such that for some $\delta >0$, 
\begin{align}
\label{pconditions}
|Q|&\leq \frac{C}{|x|^{2-\delta}} &\text{if } |x|\leq 1, d=3 \\
|Q|&\leq \frac{C}{|x|^{2}} &\text{if } |x|\leq 1, d>3 \nonumber \\
|Q|&\leq \frac{C}{|x|^{3+\delta}} &\text{if } |x|\geq 1 \nonumber \\
|\partial^{\beta} p|&\leq C|x|^{-(2+\beta+\delta)} &\text{if } |x|\geq 1 \nonumber 
\end{align}
where $\beta$ is a multi-index.

We will show the following theorem: 
\begin{theorem}
\label{Theorem_to_prove}
For $d\geq 3$, $\lambda_0>0$ big enough, let $f$ be such that $\||\cdot |^{3/2}f\|_{L^2}< \infty$, $N_1(f)< \infty$, let $V$ be a potential described by \eqref{pconditions} and $K$ the solution of \eqref{Eikonal} . Then, for any $\lambda \geq \lambda_0$, the solution $u$ of the Helmholtz equation given by \eqref{res_sol} obeys the following Sommerfeld radiation condition:
\begin{align*}
 \sup_{R\geq 1} R\int_{K\geq R} |\nabla (e^{-i\lambda^{1/2}K}u)|^2
\leq C\left[ (N_1(f))^2 + \int  |x|^3 |f|^2 \right]
\end{align*}
\end{theorem}

Our proof will largely follow \cite{barcelo_forward_2012}, but we modify the potential as described above to include a long-range potential $p$. The non-spherical approach will then allow us to control the terms of the proof stemming from $p$, giving a similar result as Theorem \ref{BVZ_Theorem}, but with a wider class of potentials. 

Indeed, while $Q$ follows the same conditions as the potential in \cite{barcelo_forward_2012}, we were able to reduce the decay necessary on $p$. In particular, our result allows for potentials that behave like $|x|^{-3}$, an important case for physical applications of the Helmholtz equation. The improved conditions also allow for potentials approaching quadratic decay; however, for potentials behaving like $|x|^{-2}$, our result does not necessarily hold.  

In the next section, we will start by showing some properties of the eikonal solution $K$.  In section 3 we will adapt the proof of Prop. 2.5 of \cite{barcelo_forward_2012} to show a key equality in our modified setting. In section 4 we use this equality to show the Sommerfeld radiation condition with our more general potential.  \\

\noindent\textbf{Acknowledgements.} We thank L. Vega for having proposed this interesting problem. We are especially grateful to J.P.G. Ramos for his advice and guidance during the writing of this paper.

\section{Properties of the eikonal equation}

Let $K$ be the solution of the eikonal equation \eqref{Eikonal}. 
We introduce the following notation:
\begin{equation*}
\nabla u = \frac{\nabla K}{| \nabla K |} \nabla_K^r u + \nabla_K^{\perp}u,
\end{equation*} 
so that $\nabla K \cdot \nabla_K^{\perp}=0$ and $|\nabla u|^2=|\nabla_K^r u|^2 + |\nabla_K^{\perp}u|^2$. 
We will use the notation $C$ for an arbitrary strictly positive constant.

Some properties of the solution of \eqref{Eikonal} have been shown by Barles \cite{barles_ceremade_eikonal_1987}. Given the decay of $p$ presented in \eqref{pconditions}, he showed that we can write $K$ as 
\begin{equation}
\label{noderiv}
K(x)=g(x)|x|
\end{equation}
where 
\begin{equation}
\label{gone}
\lim_{\lambda\rightarrow \infty}g(x)=1
\end{equation}
uniformly in $x$ and, for $1\leq i,j\leq n$,
\begin{equation}
\label{gprime}
\lim_{\lambda \rightarrow \infty}|x|| \partial_i g|=\lim_{\lambda \rightarrow \infty} |x|^2 |\partial_{ij}g|=0
\end{equation}
In particular, if we take $\lambda$ big enough, then there exist $c_0, c_1>0$ close to $1$ such that $c_0 |x|\leq K(x) \leq c_1 |x|$. Using a result by Izosaki \cite{isozaki_eikonal_1980}, we can further show that $K\in C^{\infty}$, and therefore, away from $0$, $g\in C^{\infty}$.
 
We have the following identity (see \cite{zubeldia_limiting_2013} or \cite{saito_schrodinger_1987} for the proof):
\begin{lemma}
\label{Lemma1}
The following identity holds:
\begin{equation}
\partial^2_{ij}K=\frac{|\nabla K|^2}{K}\delta_{ij}-\frac{1}{K}\partial_iK\partial_jK+\frac{1}{K}F_{ij}
\end{equation}
where 
\begin{align}
\label{Ffunction}
F_{ij}= -\delta_{ij}(|x|^2|\nabla g|^2+2gx\cdot \nabla g)+|x|^2\partial_ig\partial_jg +2x_ig\partial_jg+2x_jg\partial_ig + g|x|^2\partial^2_{ij}g
\end{align} is a bounded function.
\end{lemma}
From this we compute
\begin{align}
\label{Fgrad}
\nabla (\sum_{i=1}^nF_{ii})
=-(n-1)(\frac{\nabla p}{\lambda} - 2g\nabla g)+ 2xg\Delta g + |x|^2\nabla g \Delta g + |x|^2g\nabla \Delta g \\
+ 2\nabla g |x|\partial_r g + 2g \frac{x}{|x|}\partial_r g + 2g|x|\nabla \partial_r g. \nonumber
\end{align}

Due to the properties \eqref{Eikonal} and \eqref{noderiv}, we already have a good idea of how $K$ and $\partial_i K$ behave. The proof of Theorem \ref{Theorem_to_prove} will however also require a certain control over  $F_{ij}$. In this section, we prove decay of the derivatives of $g$, which will translate into the appropriate decay for $F$.

Expanding \eqref{Eikonal}, we get the following Hamilton-Jacobi equation:
\begin{equation}
\label{HJ}
g^2 + 2gx\cdot \nabla g + |x|^2 |\nabla g|^2 =1+\frac{p}{\lambda}.
\end{equation}

Perthame and Vega \cite{perthame_energy_2006} showed that, under similar conditions, we have the estimate 
\begin{equation}
\label{perthestimate}
|x|^{3+\delta}|\partial_r g|\leq \frac{C}{\lambda}
\end{equation}
by taking the value of $x$ that maximizes (resp. minimizes) the left side of \eqref{perthestimate}, and using that at this point, the gradient vanishes, to get an upper (resp. lower) bound on the expression.  

We will adapt his method to find similar bounds on $\partial_i g$.
We start by taking the derivative $\partial_i$ of \eqref{HJ}:
\begin{align}
\label{HJdi}
g\partial_i g +\partial_i g |x|\partial_r g + g \frac{x_i}{|x|} \partial_r g + gx \cdot \nabla \partial_{i} g + x_i|\nabla g|^2 + |x|^2\nabla g \cdot \nabla \partial_i g= \frac{\partial_ i p}{2\lambda}. 
\end{align} 
We consider the minimum point $x_0$ of  
\begin{align*}
\min |x|^{3+\delta} \partial_i g.
\end{align*}
Then, at $x_0$, we have
\begin{align*}
(3+\delta)|x|^{1+\delta} x \partial_i g + |x|^{3+\delta}\nabla (\partial_i g)=0 \\
\nabla (\partial_i g)=-(3+\delta)|x|^{-2} x \partial_i g
\end{align*}
Inserting this in \eqref{HJdi}, we get:
\begin{align*}
(2+\delta) (g+ \partial_r g   |x|) \partial_i g=  -\frac{\partial_ i p}{2\lambda}+\frac{x_i}{|x|^2}(g |x|\partial_r g + |x|^2|\nabla g|^2)
\end{align*}
By the properties \eqref{gone} and \eqref{gprime}, we have $(2+\delta) (g+ \partial_r g   |x|) \simeq 1$. 
If $x_i=0$, it is easy to see that $\partial_i g$ and $\partial_i p$ have the same decay. \\
Suppose first that $x_i>0$. Then
\begin{align*}
 \partial_i g\gtrsim \frac{ -\partial_ i p}{2\lambda}+ \frac{x_i}{|x|}g \partial_r g
\end{align*}
from which $\min |x|^{3+\delta} \partial_i g\geq -C/\lambda$ follows. \\
If $x_i<0$, we get
\begin{align*}
   \partial_i g\simeq -\frac{\partial_ i p}{2\lambda}-\frac{|x_i|}{|x|^2}(1+\frac{p}{\lambda}-g^2)
\end{align*}
In \cite{barles_ceremade_eikonal_1987} we find the estimate $1+ \inf \frac{p}{\lambda} \leq g^2 \leq 1+\sup \frac{p}{\lambda}$. Then
 \[
 (1+\frac{p}{\lambda}-g^2) \leq C|x|^{-2-\delta}, 
 \]
and we get $\min |x|^{3+\delta}\partial_i g\geq -C/\lambda$. 

If we consider the maximum point $x_1$ of $\max |x|^{3+\delta}\partial_i g$, we can use similar arguments as for the minimum case to get the bound $\max |x|^{3+\delta}\partial_i g \leq \frac{C}{\lambda} $. 
We have thus shown that, for all $x$, 
 \begin{equation}
 \label{bounddig}
 |x|^{3+\delta}|\partial_i g|\leq \frac{C}{\lambda}.
 \end{equation}

We can use this property to directly gain some more decay. 
From equation \eqref{HJdi} we have
\begin{align*}
    |x|  |\partial_r \partial_{i} g|   \lesssim \frac{|\partial_ i p|}{2\lambda}+g|\partial_i g|+|\partial_i g| |x||\partial_r g|+g |\partial_r g|+|x||\nabla g|^2+|x|^2|\nabla g|  |\nabla \partial_i g|.
\end{align*}
Using our estimates on $p$, as well as \eqref{gone}, \eqref{gprime} and \eqref{bounddig} we then conclude that 
\begin{equation}
|\partial_r\partial_i g|\leq \frac{C}{|x|^{4+\delta}}.
\end{equation}

We now want to bound $\partial^2_{ij}g$; to simplify the presentation of the proof, we assume that $\partial^2_{ij}g=0$ when $i \neq j$, the proof in the general case is similar. \\
We take the derivative $\partial^2_{ii}$ of \eqref{HJ}.
This gives us, for all $i\leq n$,
\begin{align*}
&(\partial_i g)^2 + g\partial_{ii}^2 g+\partial_{ii}^2 g |x|\partial_r g+ 2 \partial_i g \frac{x_i}{|x|}\partial_r g+2 \partial_i g |x|\partial_i\partial_r g \\
 +&g (\frac{1}{|x|}-\frac{x_i^2}{|x|^3}) \partial_r g+2g \frac{x_i}{|x|} \partial_i \partial_r g + g|x|\partial_r \partial^2_{ii} g 
+ x_i\nabla g \cdot \nabla (\partial_i g)+|\nabla g|^2 \\
+ &2x_i\nabla g \cdot \nabla \partial_i g+|x|^2\nabla \partial_i g \cdot \nabla \partial_i g+|x|^2\nabla g \cdot \nabla \partial^2_{ii} g= \frac{\partial^2_{ii} p}{2\lambda}.
\end{align*}
Let $x_0^{i}$ be the maximum  of $|x|^{4+\delta}|\partial^2_{ii}g|$. Then, at $x_0^{i}$, 
\begin{align*}
\nabla \partial^2_{ii}g=-(4+\delta)\frac{x}{|x|^2}\partial^2_{ii}g.
\end{align*}
We take the index $i_0$ which maximises the above maxima,  
\begin{equation}
\label{max_cond_equal}|x_0^{i_0}|^{4+\delta}|\partial_{i_0i_0}^2g(x_0^{i_0})|\geq |x_0^{i}|^{4+\delta}|\partial_{ii}^2g(x_0^{i})|,
\end{equation} for all $1\leq i \leq n$. From now on, we drop the indices of $x$ and $i$ to avoid confusion. \\
Then
\begin{align*}
&(3+\delta) (g+|x|\partial_r g ) |x|^{4+\delta}|\partial^2_{ii}g| \\
\leq &\frac{|x|^{4+\delta}|\partial^2_{ii} p|}{2\lambda}+|x|^{4+\delta}|\partial_i g|^2+2|x|^{4+\delta}|\partial_i g| |\partial_r g| \\
+&2|\partial_i g| |x|^{5+\delta}|\partial_i\partial_r g| +2g |x|^{3+\delta}|\partial_r g|+2g|x|^{4+\delta}  |\partial_i \partial_r g| \\
+&|x|^{4+\delta}|\nabla g|^2+3|x|^{5+\delta}|\nabla g|  |\nabla \partial_i g|+|x|^{6+\delta}|\nabla \partial_i g|^2.
\end{align*}
By our choice of index $i$, 
\begin{align*}
|x|^{4+\delta}|\nabla \partial_i g|\leq C|x|^{4+\delta} |\partial^2_{ii}g|,
\end{align*}
and we can conclude that
\begin{equation}
|\partial_{ii}^2 g |\leq \frac{C}{|x|^{4+\delta}}.
\end{equation}
  If we do not impose $\partial^2_{ij}g=0$ when $i \neq j$, we get through a similar (albeit slightly longer) proof the result
\begin{equation*}
|\partial_{ij}^2 g |\leq \frac{C}{|x|^{4+\delta}}.
\end{equation*}

 To finish our estimates on the terms of $F$, we still need to bound $\nabla \Delta g$. 
We derive the Hamilton-Jacobi equation \eqref{HJ} by $ \partial_j \Delta$ and apply the same methodology as for the previous steps to get the estimate
\begin{equation*}
|\partial_j \Delta g|\leq \frac{C}{\lambda}|x|^{-5-\delta}.
\end{equation*}

After all this work, we can finally use all these estimates on \eqref{Ffunction} and \eqref{Fgrad} to get 
\begin{align*}
|F_{ij}| \leq \frac{C}{|x|^{2+\delta}}
\end{align*}
and 
\begin{align*}
|\nabla (\sum_{i=1}^nF_{ii})|\leq& \frac{C}{|x|^{3+\delta}}.
\end{align*}

\section{Key Results}
Consider the modified Helmholtz equation 
\begin{equation}
\label{Helm}
\Delta u + p(x)u + Q(x)u + \lambda u + i \epsilon u=f.
\end{equation}
Let $K$ be the solution of the eikonal equation 
\begin{equation}
|\nabla K|^2=1 + \frac{p}{\lambda}.
\end{equation}
We assume that, along our assumptions on $p$ and $Q$ stated in the introduction, the operator $\Delta + p + Q$ is self-adjoint in $L^2(\mathbb{R}^d)$ with domain $H^1(\mathbb{R}^d)$. 

In \cite{barcelo_forward_2012}, the authors prove Theorem \ref{BVZ_Theorem} by working on the sphere of radius $|x|=R$. The strategy of this paper is to follow the idea of Saito \cite{saito_schrodinger_1987} and work on the surface defined by $K(x)=R$, with normal vector $\frac{\nabla K}{|\nabla K|}$. Note that by taking $p=0$, we have $K=|x|$ and obtain the results of \cite{barcelo_forward_2012}.

The following lemma can easily be proven:
\begin{lemma}
The solution $u$ of the Helmholtz equation (\ref{Helm}) satisfies
\begin{align}
\label{Lemma2a}
\int \phi \lambda(1+\frac{p}{\lambda}) |u|^2 - \int \phi |\nabla u|^2 + \int \phi Q |u|^2 -\mathcal{R}\int \nabla \phi \cdot \nabla u \overline{u}
=\mathcal{R}\int \phi f \overline{u} \\
\label{Lemma2b}
\epsilon \int \phi |u|^2-\Im \int\nabla \phi \cdot \nabla u \overline{u}=\Im \int \phi f \overline{u}.
\end{align} 
\end{lemma} 
We also have the following:
\begin{lemma}
Let $\psi$ be a $K$-radial function and regular. Then any solution $u$ of (\ref{Helm}) satisfies 
\begin{align}
&\Re \int  \nabla u \cdot (D^2\psi \cdot \nabla \overline{u} ) +\frac{1}{2}\Re\int \nabla u \cdot (\nabla (\Delta \psi) \overline{u}) \nonumber \\
\label{Lemma3}
+&  \frac{1}{2}\Re\int \nabla \psi \cdot (\nabla p|u|^2) + \epsilon \Im \int  \nabla \psi \cdot \nabla \overline{u}u \\
-& \Re\int Q(x)u \nabla \psi \cdot \nabla \overline{u}u -\frac{1}{2}\Re\int Q(x)\Delta \psi |u|^2 \nonumber \\
=-&\Re\int f \nabla \psi \cdot \nabla \overline{u} -\frac{1}{2}\Re\int f\Delta \psi \overline{u} \nonumber
\end{align}
\end{lemma}
The idea of the proof is to multiply \eqref{Helm} by $\nabla \psi \cdot \nabla \overline{u} +\frac{1}{2}\Delta \psi \overline{u}$, integrate and take the real part. Then, after integrating by parts, we get the result. For more details see \cite{perthame_morreycampanato_1999}.

Combining the two lemmata in a precise way, we get the following key proposition:

\begin{lemma}
\label{Keylemma}
If $\psi$ is $K$-radial, then any solution $u$ of the Helmholtz equation (\ref{Helm}) satisfies
\begin{align*}
&\frac{1}{2}\int |\nabla K|^2 \psi''\left|\nabla_K^r u-i\lambda^{1/2}|\nabla K| u\right|^2
+\int \left(\frac{\psi'}{K}-\frac{\psi''}{2}\right)|\nabla K|^2|\nabla_K^{\perp} u|^2\\
&+ \Re \frac{d-1}{2}\int \nabla \left( \frac{\psi'}{K}|\nabla K|^2\right) \cdot\nabla u \overline{u}\\
&-\frac{d-1}{2}\int  \frac{\psi'}{K}|\nabla K|^2  Q |u|^2 -\Re \int Q\nabla \psi \cdot \nabla \overline{u} u\\
&-\Im \lambda^{1/2} \int\nabla (|\nabla K|^2)\psi'\cdot \nabla u \overline{u}+  \frac{1}{2}\int \nabla \psi \cdot \nabla p|u|^2\\
&+\int\frac{\psi'}{K}\left(\nabla u \cdot F \cdot \nabla_K^{\perp}\overline{u}+ \nabla u \cdot F \cdot \frac{\nabla K}{|\nabla K|}\nabla_K^ru\right)\\
&+ \Re \frac{1}{2}\int \nabla \left( \frac{\psi'}{K}\sum F_{ii}\right) \cdot\nabla u \overline{u}
-\frac{1}{2}\int \left( \frac{\psi'}{K}\sum F_{ii} \right) Q |u|^2\\
&+\frac{\epsilon}{2\lambda^{1/2}}\int \psi' \left| \nabla u-i\lambda^{1/2} \nabla K u\right|^2\\
&-\frac{\epsilon}{2\lambda^{1/2}}\int \psi' Q |u|^2 +\frac{\epsilon}{2\lambda^{1/2}} \Re  \int \nabla (\psi') \cdot \nabla u \overline{u}\\
=&-\Re \int  f \nabla \psi \cdot \overline{\nabla u}-\frac{1}{2} \Re \int f \left[ \psi'\left(\frac{d-1}{K}|\nabla K|^2 +\frac{1}{K}\sum F_{ii}\right)\right]\overline{u}\\
&+\Im \lambda^{1/2} \int |\nabla K|^2 \psi' f \overline{u}-\frac{\epsilon}{2\lambda^{1/2}}\Re\int  \psi' f \overline{u}. 
\end{align*}
\end{lemma}
\begin{proof}

We calculate \eqref{Lemma3} + \eqref{Lemma2a} + $\lambda^{1/2}$ \eqref{Lemma2b}, where we replace $\phi$ by $\frac{1}{2}|\nabla K|^2\psi''$ in \eqref{Lemma2a} and by $|\nabla K|^2 \psi'$ in \eqref{Lemma2b}
\begin{align}
&\Re \int \nabla u \cdot D^2 \psi \cdot \overline{\nabla u}+ \Re \frac{1}{2}\int \nabla (\Delta \psi) \cdot\nabla u \overline{u} + \epsilon \Im \int \nabla \psi \cdot \nabla \overline{u} u \nonumber \\
-&\frac{1}{2}\int \Delta \psi Q |u|^2 -\Re \int Q\nabla \psi \cdot \nabla \overline{u}u +  \frac{1}{2}\int \nabla \psi \cdot \nabla p|u|^2 \nonumber \\
+&\int \frac{1}{2}|\nabla K|^2\psi'' \lambda(1+\frac{p}{\lambda}) |u|^2 - \int \frac{1}{2}|\nabla K|^2\psi'' |\nabla u|^2 \nonumber \\ 
\label{sum_lemmata}
+& \int \frac{1}{2}|\nabla K|^2\psi'' Q |u|^2 -\Re\int \nabla (\frac{1}{2}|\nabla K|^2\psi'') \cdot \nabla u \overline{u}  \\
+&\epsilon \lambda^{1/2} \int |\nabla K|^2\psi' |u|^2-\Im \lambda^{1/2} \int\nabla (|\nabla K|^2\psi') \cdot \nabla u \overline{u} \nonumber \\
=-&\Re \int  f\nabla \psi \cdot \overline{\nabla u}-\frac{1}{2} \Re \int f \Delta \psi \overline{u} \nonumber \\
+&\Re \int \frac{1}{2}|\nabla K|^2\psi''f \overline{u}+\Im \lambda^{1/2} \int |\nabla K|^2 \psi' f \overline{u}. \nonumber
\end{align}

We have the following result:
\begin{align*}
\nabla u \cdot D^2 \psi \cdot \nabla \overline{u}
=&\frac{\psi'}{K}|\nabla K|^2|\nabla_K^{\perp} u|^2 + |\nabla K|^2 \psi''|\nabla _K^r u|^2
+\frac{\psi'}{K}\nabla u \cdot F \cdot \nabla \overline{u}.
\end{align*}
Also
\begin{align*}
\Delta \psi=\psi''|\nabla K|^2+ \psi'\left(\frac{d-1}{K}|\nabla K|^2 +\frac{1}{K}\sum F_{ii}\right).
\end{align*}
We can explicit the following term:
\begin{align*}
\nabla (|\nabla K|^2\psi')=\nabla (|\nabla K|^2)\psi'+  \psi'' \nabla K|\nabla K|^2.
\end{align*}
We rewrite the $\psi''$ terms in the following way:
\begin{align*}
&\int|\nabla K|^2 \psi''|\nabla _K^r u|^2+\frac{1}{2}\int |\nabla K|^4\psi'' \lambda |u|^2\\
&- \frac{1}{2}\int |\nabla K|^2\psi'' (|\nabla^r_K u|^2+|\nabla_K^{\perp} u|^2 )-\Im \lambda^{1/2} \int \psi'' \nabla K|\nabla K|^2\cdot \nabla u \overline{u}\\
&=\frac{1}{2}\int |\nabla K|^2 \psi''(|\nabla_K^r u|^2+\lambda |\nabla K|^2|u|^2)-\Im \lambda^{1/2} \int \psi'' |\nabla K|^3 \nabla_K^r u \overline{u}\\
&-\frac{1}{2}\int |\nabla K|^2\psi'' |\nabla_K^{\perp}u|^2.
\end{align*}

We now subtract \eqref{Lemma2a} from \eqref{sum_lemmata}, replacing $\phi$ by $\frac{\epsilon}{2\lambda^{1/2}}\psi'$:
\begin{align}
&\frac{1}{2}\int |\nabla K|^2 \psi''(|\nabla_K^r u|^2+\lambda|\nabla K|^2 |u|^2)-\Im \lambda^{1/2} \int \psi'' |\nabla K|^3 \nabla_K^r u \overline{u} \nonumber \\
&-\frac{1}{2}\int |\nabla K|^2\psi'' |\nabla_K^{\perp}u|^2+\int\frac{\psi'}{K}\left(\nabla u \cdot F \cdot \nabla\overline{u} \right) \nonumber \\
&+\int \frac{\psi'}{K}|\nabla K|^2|\nabla_K^{\perp} u|^2+ \Re \frac{d-1}{2}\int \nabla \left( \frac{\psi'}{K}|\nabla K|^2\right) \cdot\nabla u \overline{u} \nonumber  \\
&+\epsilon \lambda^{1/2}\int |\nabla K|^2\psi' |u|^2+ \epsilon \Im \int \nabla \psi \cdot \nabla \overline{u} u \nonumber \\
&-\frac{d-1}{2}\int  \frac{\psi'}{K}|\nabla K|^2  Q |u|^2 -\Re \int Q\nabla \psi \cdot \nabla \overline{u} u \nonumber \\
\label{lemma4_middle}
&-\Im \lambda^{1/2} \int\nabla (|\nabla K|^2)\psi'\cdot \nabla u \overline{u} +  \frac{1}{2}\int \nabla \psi \cdot \nabla p|u|^2\\
&+ \Re \frac{1}{2}\int \nabla \left( \frac{\psi'}{K}\sum F_{ii}\right) \cdot\nabla u \overline{u}
-\frac{1}{2}\int \left( \frac{\psi'}{K}\sum F_{ii} \right) Q |u|^2 \nonumber \\
&-\int \frac{\epsilon}{2\lambda^{1/2}}\psi' (\lambda+p+Q) |u|^2 + \int \frac{\epsilon}{2\lambda^{1/2}}\psi' |\nabla u|^2 \nonumber \\
&+\Re\int\frac{\epsilon}{2\lambda^{1/2}} \nabla (\psi') \cdot \nabla u \overline{u}\nonumber \\
=&-\Re \int  f\nabla \psi \cdot \overline{\nabla u}-\frac{1}{2} \Re \int f \left[ \psi'\left(\frac{d-1}{K}|\nabla K|^2 +\frac{1}{K}\sum F_{ii}\right)\right]\overline{u}\nonumber \\
&+\Im \lambda^{1/2} \int |\nabla K|^2 \psi' f \overline{u}-\Re\int \frac{\epsilon}{2\lambda^{1/2}} \psi' f \overline{u}. \nonumber 
\end{align}
As 
\begin{equation*}
|\nabla_K^ru|^2+\lambda|\nabla K|^2 |u|^2 -2\Im \lambda^{1/2} |\nabla K|\nabla_K^r u\overline{u}=|\nabla_K^ru -i\lambda^{1/2}|\nabla K| u|^2,
\end{equation*}
we can write the first line as 
\begin{equation*}
\frac{1}{2}\int |\nabla K|^2 \psi''\left|\nabla_K^r u-i\lambda^{1/2}|\nabla K| u\right|^2.
\end{equation*}
We simplify the terms with $\epsilon$:
\begin{align*}
&\frac{\epsilon}{2\lambda^{1/2}}2\int |\nabla K|^2\psi' \lambda|u|^2+ \frac{\epsilon}{2\lambda^{1/2}} \Im \int 2\lambda^{1/2} \psi' \nabla K \cdot \nabla \overline{u} u \\
&-\int \frac{\epsilon}{2\lambda^{1/2}}\psi' (\lambda|\nabla K|^2+Q) |u|^2 + \int \frac{\epsilon}{2\lambda^{1/2}}\psi' |\nabla u|^2 \\
&+\Re\int\frac{\epsilon}{2\lambda^{1/2}} \nabla (\psi') \cdot \nabla u \overline{u}.
\end{align*}
 Then, since
\begin{align*}
\left| \nabla u-i\lambda^{1/2} \nabla K u\right|^2=|\nabla u|^2 + \lambda |\nabla K|^2 -2\lambda^{1/2}\Im \nabla K \cdot \nabla \overline{u} u,
\end{align*}
this becomes
\begin{align*}
&\frac{\epsilon}{2\lambda^{1/2}}\int \psi'\left[|\nabla K|^2 \lambda|u|^2- \Im 2\lambda^{1/2}  \nabla K \cdot \nabla \overline{u} u+|\nabla u|^2\right] \\
-&\frac{\epsilon}{2\lambda^{1/2}}\int \psi' Q |u|^2 +\Re\int\frac{\epsilon}{2\lambda^{1/2}} \nabla (\psi') \cdot \nabla u \overline{u}\\
=&\frac{\epsilon}{2\lambda^{1/2}}\int \psi' \left| \nabla u-i\lambda^{1/2} \nabla K u\right|^2\\
-&\frac{\epsilon}{2\lambda^{1/2}}\int \psi' Q |u|^2 +\Re\int\frac{\epsilon}{2\lambda^{1/2}} \nabla (\psi') \cdot \nabla u \overline{u}.
\end{align*}
The result thus follows from \eqref{lemma4_middle}.
\end{proof}

\section{Main result}

Armed with Lemma \ref{Keylemma}, we are now ready to show the following proposition, from which our main theorem will follow. 
\begin{proposition}
\label{prop1}
Let $d\geq 3$, $\epsilon >0$ and $f$ such that $N_1(f)<\infty$ and $\||\cdot|^2 f\|_{L^2}< \infty$. Then there exist $\lambda_0 >0$ and a constant $C=C(\lambda_0)$ such that, for any $R\geq 1$ and $\lambda \geq \lambda_0$, the solution $u\in H^1(\mathbb{R}^d)$ to the Helmholtz equation \eqref{Helm} satisfies 
\begin{align*}
&\int_{K\leq R} |\nabla K|^2 K \left|\nabla_K^r u-i\lambda^{1/2}|\nabla K| u+\frac{d-1}{2K}|\nabla K|u\right|^2 + R\int_{K\geq 2R} |\nabla K|^2|\nabla (e^{-i\lambda^{1/2}K}u)|^2 \\
&+\frac{\epsilon}{2\lambda^{1/2}}\int_{K\leq R} K^2 \left| \nabla u-i\lambda^{1/2} \nabla K u\right|^2
+\frac{\epsilon R}{2\lambda^{1/2}}\int_{K \geq 2R} K \left| \nabla u-i\lambda^{1/2} \nabla K u\right|^2\\
\leq &C\left[|||u|||_1^2+|||\nabla u|||_1^2 +\frac{N_1(f)^2}{\lambda} + \int (1+\epsilon |x|) |x|^3 |f|^2 \right]
\end{align*}
\end{proposition}
We recall the definitions of the Agmon-Hörmander norms:
\begin{equation*}
|||u|||_{R_0}:=\sup_{R\geq R_0}\left(\frac{1}{R}\int_{|x|\leq R}|u(x)|^2\right)^{1/2}
\end{equation*}
and 
\begin{align*}
N_{R_0}(u):=\sum_{j>J}\left( 2^{j+1}\int_{C(j)}|u|^2 \right)^{1/2} + \left( R_0\int_{|x| \leq R_0} |u|^2 \right)^{1/2}
\end{align*}
where $J$ is such that $2^{J-1} < R_0 < 2^J$ and $C(j)=\lbrace x \in \mathbb{R}^d : 2^j \leq |x| \leq 2^{j+1}\rbrace$. Note that for $R_0 \geq 0$
\[
\left| \int uv \right| \leq |||u|||_{R_0}N_{R_0}(v).
\]

\begin{proof}
Let $R\geq 1$ and observe that 
\begin{align*}
|||u|||_R^2 &\geq \frac{1}{2R}\int_{K\leq 2R}|u|^2 \geq \frac{1}{2R}\int_{R \leq K \leq 2R}|u|^2=\frac{1}{2R}\int_R^{2R}\frac{1}{r}\int_{K=r}|x||u|^2\\ 
&\geq \frac{\log2}{2}\inf_{R \leq r \leq 2R}\int_{K=r}|u|^2,
\end{align*}
which implies that there exists $R_1$ such that $R\leq R_1 \leq 2R$ and 
\begin{align}
\label{equalnok}
\int_{|K|=R_1}|u|^2 \leq C |||u|||_{R_1}^2, \\
\label{equalk}
\int_{|K|=R_1}K|u|^2 \leq C \int |u|^2.
\end{align} 
Then we have the following: for $R_0>1$,
\begin{align*}
\int_{K\leq R_0}|u|^2
=\int_{K\leq1}|u|^2+\int_{1< K\leq R_0}|u|^2. 
\end{align*}
It is easy to see that 
\begin{align*}
\int_{K\leq1}|u|^2 \leq C |||u|||_1^2
\end{align*}
and 
\begin{align*}
\int_{1\leq K\leq R_0}\frac{|u|^2}{|x|^\alpha} \leq C|||u|||_1^2 \\
\int_{K\geq R_0}\frac{|u|^2}{|x|^\alpha} \leq C|||u|||_1^2
\end{align*}
when $\alpha>1$. 

Let $\psi'(K(x))=K^2$ if $K\leq R_1$ and $\psi'(K(x))=R_1 K$ if $K\geq R_1$. Then (in a distributional sense), $\psi''=2K$ for $K\leq R_1$ and $\psi''=R_1$ for $K\geq R_1$. In the following, will use the notation $\int_{\leq }:=\int_{K\leq R_1 }$ and $\int_{\geq }:=\int_{K\geq R_1 }$. We put this $\psi$ int the equation of Lemma \ref{Keylemma}, and, after some simplification, we get:
\begin{align}
\label{A}&\int_{\leq} |\nabla K|^2 K \left|\nabla_K^r u-i\lambda^{1/2}|\nabla K| u\right|^2
+\frac{R_1}{2}\int_{\geq} |\nabla K|^2\left|\nabla_K^r u-i\lambda^{1/2}|\nabla K| u\right|^2\\
\label{B}&+ \frac{R_1}{2}\int_{\geq}|\nabla K|^2|\nabla_K^{\perp} u|^2 
 \\
\label{C}&+ \Re \frac{d-1}{2}\int_{\leq} \nabla \left( K|\nabla K|^2\right) \cdot\nabla u \overline{u}
+ \Re \frac{d-1}{2}R_1\int_{\geq} \nabla \left( |\nabla K|^2\right) \cdot\nabla u \overline{u}\\
\label{D}&-\frac{d-1}{2}\int_{\leq}  K|\nabla K|^2  Q |u|^2 
-\frac{d-1}{2} R_1 \int_{\geq} |\nabla K|^2  Q |u|^2 \\
\label{D.5}&+  \frac{1}{2}\int_{\leq} K^2 \nabla K \cdot \nabla p|u|^2
+  \frac{1}{2}\int_{\geq} R_1 K\nabla K \cdot \nabla p|u|^2\\
\label{E}&-\Re \int_{\leq} QK^2 \nabla K \cdot \nabla \overline{u} u 
-\Re \int_{\geq} Q R_1 K \nabla K \cdot \nabla \overline{u} u \\
\label{F}&-\Im \lambda^{1/2} \int_{\leq}\nabla (|\nabla K|^2)K^2\cdot \nabla u \overline{u}
-\Im \lambda^{1/2} \int_{\geq}\nabla (|\nabla K|^2)R_1K\cdot \nabla u \overline{u}\\
\label{G}&+\int_{\leq} K\left(\nabla u \cdot F \cdot \nabla\overline{u}\right) +\int_{\geq}R_1\left(\nabla u \cdot F \cdot \nabla \overline{u}\right)\\
\label{I}&+ \Re \frac{1}{2}\int_{\leq} \nabla \left( K\sum F_{ii}\right) \cdot\nabla u \overline{u}
+ \Re \frac{1}{2}\int_{\geq} \nabla \left( R_1\sum F_{ii}\right) \cdot\nabla u \overline{u} \\
\label{J}&-\frac{1}{2}\int_{\leq} \left( K\sum F_{ii} \right) Q |u|^2
-\frac{1}{2}\int_{\geq} \left(R_1\sum F_{ii} \right) Q |u|^2\\
\label{K}&+\frac{\epsilon}{2\lambda^{1/2}}\int_{\leq} K^2 \left| \nabla u-i\lambda^{1/2} \nabla K u\right|^2
+\frac{\epsilon}{2\lambda^{1/2}}\int_{\geq} R_1K \left| \nabla u-i\lambda^{1/2} \nabla K u\right|^2\\
\label{L}&-\frac{\epsilon}{2\lambda^{1/2}}\int_{\leq} K^2 Q |u|^2
-\frac{\epsilon}{2\lambda^{1/2}}\int_{\geq} R_1K Q |u|^2 \\
\label{M}& +\frac{\epsilon}{2\lambda^{1/2}} \Re  \int_{\leq} 2K \nabla K \cdot \nabla u \overline{u}
+\frac{\epsilon}{2\lambda^{1/2}} \Re  \int_{\geq} R_1 \nabla K \cdot \nabla u \overline{u}\\
\label{N}=&-\Re \int_{\leq}  f K^2 \nabla K \cdot \overline{\nabla u}
-\Re \int_{\geq}  f R_1 K\nabla K \cdot \overline{\nabla u} \\
\label{O}&-\frac{1}{2} \Re \int_{\leq} f \left[ K\left((d-1)|\nabla K|^2 +\sum F_{ii}\right)\right]\overline{u} \\
\label{P}&-\frac{1}{2} \Re \int_{\geq} f \left[ R_1 \left((d-1)|\nabla K|^2 +\sum F_{ii}\right)\right]\overline{u}\\
\label{Q}&+\Im \lambda^{1/2} \int_{\leq} |\nabla K|^2 K^2 f \overline{u}
+\Im \lambda^{1/2} \int_{\geq} |\nabla K|^2 R_1K f \overline{u}\\
\label{R}&-\frac{\epsilon}{2\lambda^{1/2}}\Re\int_{\leq}  K^2 f \overline{u} 
-\frac{\epsilon}{2\lambda^{1/2}}\Re\int_{\geq} R_1 K f \overline{u}
\end{align}
The $\int_{\geq}$ terms of \eqref{A} and \eqref{B} can be added in the following way: since
\begin{align*}
&|\nabla_K^r u-i\lambda^{1/2}|\nabla K|u|^2+|\nabla_K^{\perp}u|^2\\
=&|\nabla_K^r u|^2+|\nabla_K^{\perp}u|^2+\lambda|\nabla K|^2 |u|^2-2\Im \lambda^{1/2}|\nabla K|\nabla_K^ru \overline{u}\\
=&|\nabla u|^2+\lambda |\nabla K|^2|u|^2-2\Im \lambda^{1/2} \nabla K \cdot \nabla u \overline{u}=|\nabla u-i\lambda^{1/2}\nabla K u|^2,
\end{align*} we have
\begin{align*}
\frac{R_1}{2}\int_{\geq} |\nabla K|^2 |\nabla_K^r u-i\lambda^{1/	2}u|^2 + \frac{R_1}{2}\int_{\geq} |\nabla K|^2 |\nabla_K^{\perp}u|^2\\
=\frac{R_1}{2}\int_{\geq} |\nabla K|^2 |\nabla u-i\lambda^{1/2}\nabla K u|^2.
\end{align*}

We calculate for \eqref{C}
\begin{align*}
\nabla \left( K|\nabla K|^2\right) =\nabla K |\nabla K|^2 + K \nabla (|\nabla K|^2)=
\nabla K |\nabla K|^2 + \frac{K}{\lambda} \nabla p.
\end{align*}

Then
\begin{align*}
\Re \frac{d-1}{2}\int_{\leq} (\nabla K |\nabla K|^2)\cdot\nabla u \overline{u}=\Re\frac{1}{2} \frac{d-1}{2}\int_{\leq} (\nabla K |\nabla K|^2)\cdot\nabla |u|^2 \\
=2\Re\frac{d-1}{2}\int_{\leq} (\nabla K |\nabla K|^2)\cdot\nabla u\overline{u}-\frac{1}{2} \frac{d-1}{2}\int_{\leq} (\nabla K |\nabla K|^2)\cdot\nabla |u|^2.
\end{align*}
We integrate the second term by parts:
\begin{align*}
-\frac{1}{2} \frac{d-1}{2}\int_{\leq} (\nabla K |\nabla K|^2)\cdot\nabla |u|^2\\
=-\frac{1}{2} \frac{d-1}{2}\int_{=} (\nabla K |\nabla K|^2)|u|^2+\frac{1}{2} \frac{d-1}{2}\int_{\leq} (\Delta K |\nabla K|^2+\nabla K \nabla |\nabla K|^2)|u|^2
\end{align*}
Since Lemma \ref{Lemma1} gives 
\begin{equation*}
\Delta K=\frac{d-1}{K}|\nabla K|^2 + \frac{1}{K}\sum F_{ii},
\end{equation*}
we have for \eqref{C} 
\begin{align*}
&\Re \frac{d-1}{2}\int_{\leq} (\nabla K |\nabla K|^2)\cdot\nabla u \overline{u}\\
=&2\Re\frac{d-1}{2}\int_{\leq} (\nabla K |\nabla K|^2)\cdot\nabla u\overline{u} -\frac{1}{2} \frac{n-1}{2}\int_{=} (\nabla K |\nabla K|^2)|u|^2\\
&+\frac{1}{2} \frac{d-1}{2}\int_{\leq} (\frac{d-1}{K}|\nabla K|^2 + \frac{1}{K}\sum F_{ii} |\nabla K|^2+\nabla K \nabla |\nabla K|^2)|u|^2\\
=&2\Re\frac{d-1}{2}\int_{\leq}|\nabla K|^2K\frac{1}{K} (\nabla K )\cdot\nabla u\overline{u} -\frac{1}{2} \frac{d-1}{2}\int_{=} (\nabla K |\nabla K|^2)|u|^2\\
&+\frac{d-1}{2} \frac{d-1}{2}\int_{\leq} K\frac{1}{K^2}|\nabla K|^4|u|^2+\frac{1}{2} \frac{d-1}{2}\int_{\leq} (\frac{1}{K}\sum F_{ii} |\nabla K|^2+\nabla K \nabla |\nabla K|^2)|u|^2.
\end{align*}
The first and third term of the last expression can be combined with the first terms of \eqref{A} to form 
\begin{align*}
\int_{\leq} |\nabla K|^2 K \left|\nabla_K^r u-i\lambda^{1/2}|\nabla K| u+\frac{d-1}{2K}|\nabla K|u \right|^2.
\end{align*}

Then we can rewrite the big equation as
\begin{align}
\label{A3}&\int_{\leq} |\nabla K|^2 K \left|\nabla_K^r u-i\lambda^{1/2}|\nabla K| u+\frac{d-1}{2K}|\nabla K|u\right|^2\\
\label{A3.5}&+\frac{R_1}{2}\int_{\geq} |\nabla K|^2|\nabla u-i\lambda^{1/2}\nabla K u|^2 \\
\label{K3}&+\frac{\epsilon}{2\lambda^{1/2}}\int_{\leq} K^2 \left| \nabla u-i\lambda^{1/2} \nabla K u\right|^2
+\frac{\epsilon}{2\lambda^{1/2}}\int_{\geq} R_1K \left| \nabla u-i\lambda^{1/2} \nabla K u\right|^2\\
\label{B3}=&
-\frac{1}{2} \frac{d-1}{2\lambda}\int_{\leq} (\nabla K\cdot \nabla p)|u|^2\\
\label{C3}&- \Re \frac{d-1}{2\lambda}\int_{\leq} ( K \nabla p) \cdot\nabla u \overline{u}
-\Re \frac{d-1}{2}R_1\int_{\geq} \nabla \left(\frac{p}{\lambda}\right) \cdot\nabla u \overline{u}\\
\label{D3.5}&-  \frac{1}{2}\int_{\leq} K^2 \nabla K \cdot \nabla p|u|^2
-  \frac{1}{2}\int_{\geq} R_1 K\nabla K \cdot \nabla p|u|^2 \\
\label{D3}&+\frac{d-1}{2}\int_{\leq}  K|\nabla K|^2  Q |u|^2 
+\frac{d-1}{2} R_1 \int_{\geq} |\nabla K|^2  Q |u|^2 \\
\label{E3}&+\Re \int_{\leq} QK^2 \nabla K \cdot \nabla \overline{u} u 
+\Re \int_{\geq} Q R_1 K \nabla K \cdot \nabla \overline{u} u \\
\label{F3}&+\Im \lambda^{1/2} \int_{\leq}\nabla (|\nabla K|^2)K^2\cdot \nabla u \overline{u}
+\Im \lambda^{1/2} \int_{\geq}\nabla (|\nabla K|^2)R_1K\cdot \nabla u \overline{u}\\
\label{G3}&-\int_{\leq} K\left(\nabla u \cdot F \cdot \nabla\overline{u}\right)
-\int_{\geq}R_1\left(\nabla u \cdot F \cdot \nabla \overline{u} \right)\\
\label{I3}&- \Re \frac{1}{2}\int_{\leq} \nabla \left( K\sum F_{ii}\right) \cdot\nabla u \overline{u}
- \Re \frac{1}{2}\int_{\geq} \nabla \left( R_1\sum F_{ii}\right) \cdot\nabla u \overline{u} \\
\label{J3}&+\frac{1}{2}\int_{\leq} \left( K\sum F_{ii} \right) Q |u|^2
+\frac{1}{2}\int_{\geq} \left(R_1\sum F_{ii} \right) Q |u|^2\\
\label{J3.5}&-\frac{d-1}{4}\int_{\leq} (\frac{1}{K}\sum F_{ii} |\nabla K|^2)|u|^2+ \frac{d-1}{4}\int_{=} (\nabla K |\nabla K|^2)|u|^2\\
\label{O3}&-\frac{1}{2} \Re \int_{\leq} f \left[ K\sum F_{ii}\right]\overline{u} 
-\frac{R_1}{2} \Re \int_{\geq} f \left[ \sum F_{ii}\right]\overline{u}\\
\label{M3}& -\frac{\epsilon}{2\lambda^{1/2}} \Re  \int_{\leq} 2K \nabla K \cdot \nabla u \overline{u}
-\frac{\epsilon}{2\lambda^{1/2}} \Re  \int_{\geq} R_1 \nabla K \cdot \nabla u \overline{u}\\
\label{L3}&+\frac{\epsilon}{2\lambda^{1/2}}\int_{\leq} K^2 Q |u|^2
+\frac{\epsilon}{2\lambda^{1/2}}\int_{\geq} R_1K Q |u|^2 \\
\label{N3}&-\Re \int_{\leq}  f K^2( |\nabla K| \nabla_K^r \overline{ u}+i\lambda^{1/2}|\nabla K|^2\overline{u}+\frac{d-1}{2K}|\nabla K|^2\overline{u})\\
&\label{N3.5}-\Re \int_{\geq}  f R_1 K(|\nabla K|\nabla_K^r \overline{u}+ i\lambda^{1/2}|\nabla K|^2\overline{u}+\frac{d-1}{2K}|\nabla K|^2\overline{u})\\
\label{R3}&-\frac{\epsilon}{2\lambda^{1/2}}\Re\int_{\leq}  K^2 f \overline{u} 
-\frac{\epsilon}{2\lambda^{1/2}}\Re\int_{\geq} R_1 K f \overline{u}
\end{align}.

We are now ready to approximate the RHS of this big equation. We first notice that, since $K \simeq |x|$ and $|\nabla K|^2 \leq C$, the terms \eqref{D3}, \eqref{E3} and \eqref{L3}-\eqref{R3} can be estimated as in the proof of proposition 4.4 of \cite{barcelo_forward_2012}. 

We start with the $\epsilon$ term \eqref{M3}:
\begin{align*}
|\frac{\epsilon}{2\lambda^{1/2}} \Re  \int_{\leq} 2K \nabla K \cdot \nabla u \overline{u}|=|\frac{\epsilon}{2\lambda^{1/2}} \int_{=} K \nabla K|u|^2-\frac{\epsilon}{2\lambda^{1/2}} \Re  \int_{\leq}\nabla (K \nabla K) |u|^2|.
\end{align*} 
The $\int_=$ part can be estimated using \eqref{equalk}. Since
\begin{align*}
|\nabla (K\nabla K)|=||\nabla K|^2+K\Delta K |=| d+d\frac{p}{\lambda}+\sum F_{ii}|\leq C
\end{align*}
we have
\begin{align*}
|\frac{\epsilon}{2\lambda^{1/2}} \Re  \int_{\leq}\nabla( K \nabla K) |u|^2|\leq \frac{C\epsilon}{\lambda^{1/2}} \int_{\leq} |u|^2\leq \frac{C}{\lambda^{1/2}}N_1(f)|||u|||_1.
\end{align*}
The $\int_{\geq}$ term of \eqref{M3} is treated the same way. \\
For \eqref{B3}, 
\begin{align*}
|\int_{\leq} (\nabla K\cdot \nabla p)|u|^2 |  \leq C \int_{\leq 1} |u|^2 + C \int_{\geq 1} \frac{|u|^2}{|x|^{3+\alpha}} \leq  C|||u|||_1^2.
\end{align*}
Then, for \eqref{C3},
\begin{align*}
\left|\int ( K \nabla p) \cdot\nabla u \overline{u}\right| \leq C \int_{\leq 1}|\nabla u||u| 
+C\int_{\geq 1}\frac{|\nabla u| |u|}{|x|^{2+\alpha}} \leq C|||\nabla u|||_1^2 + C|||u|||_1^2.
\end{align*}
Also, for \eqref{D3.5},
\begin{align*}
|\int K^2 \nabla K \cdot \nabla p|u|^2|\leq C\int_{\leq 1}|K||u|^2 + C\int_{\geq 1}\frac{|u|^2}{|x|^{1+\alpha}}\leq C |||u|||_1^2.
\end{align*}
We have for \eqref{F3},
\begin{align*}
|\int \nabla (|\nabla K|^2)K^2\cdot \nabla u \overline{u}|\leq C \int|\nabla p| K^2 |\nabla u ||u| \\
\leq C \int_{\leq 1}|\nabla u ||u| + C\int_{\geq 1} \frac{|\nabla u||u|}{|x|^{1+\alpha}}\leq C |||\nabla u|||_1^2 +C |||u|||_1^2.
\end{align*}
We now work on the terms containing $F_{ij}$ using that $F_{ij}$ is bounded and, for $|x|\geq 1$, $F_{ij}\leq \frac{C}{|x|^{2+\alpha}}$.
For \eqref{G3},
\begin{align*}
\left|\int K\left(\nabla u \cdot F \cdot \nabla\overline{u}\right)\right|\leq C\int_{\leq 1}|\nabla u|^2 + C\int_{\geq 1}\frac{|\nabla u|^2}{|x|^{1+\alpha}} \leq C|||\nabla u|||_1^2.
\end{align*}
Then, for \eqref{J3},
\begin{align*}
\left|\int \left( K\sum F_{ii} \right) Q |u|^2 \right|\leq  C\int_{\leq} |K||Q||u|^2,
\end{align*} and we conclude as for \eqref{D3}. \\
For \eqref{J3.5} we get,
\begin{align*}
\left| \int_{\leq} (\frac{1}{K}\sum F_{ii} |\nabla K|^2)|u|^2\right| \leq
\left | \int_{\leq 1} (\frac{1}{K}\sum F_{ii} |\nabla K|^2)|u|^2+\int_{\geq 	} (\frac{1}{K}\sum F_{ii} |\nabla K|^2)|u|^2 \right| \\
\leq C \int_{\leq 1} \frac{K}{K^2}|u|^2+ C \int_{\geq 1	} \frac{|u|^2}{|K|^{3+\alpha}}\leq C|||\nabla u|||_1^2+ C|||u|||_1^2
\end{align*}
Since $\nabla \sum F_{ii}\leq \frac{C}{|x|^{3+\alpha}}$, \eqref{I3} gives
\begin{align*}
|\Re \int_{\leq} \nabla \left( K\sum F_{ii}\right) \cdot\nabla u \overline{u}|\leq \int_{\leq}|\nabla K| |\sum F_{ii}||\nabla u||u| + \int_{\leq}K |\nabla (\sum F_{ii})| |\nabla u||u| \\ 
\leq C|||\nabla u|||_1^2 + C |||u|||_1^2.
\end{align*} \\
For \eqref{O3}, we notice that, since $\lambda \geq \lambda_0 >0$,
\begin{align*}
\left| \Re \int_{\leq} f \left( K\sum F_{ii}\right) \overline{u} \right|\leq C\int_{\leq 1} |f||K||u| + \int_{\geq 1}|f||K|^{-(1+\alpha)}|u| \\ \leq \frac{C}{\lambda^{1/2}}\int_{\leq 1} |f||u|+ \int_{\geq 1}|f||K|^{3/2}|\frac{|u|}{|K|^{1+\alpha}} \\
\leq C\frac{(N_1(f))^2}{\lambda}+ C|||u|||_1^2+C\int|f|^2|K|^{3} + C|||u|||_1^2
\end{align*}
For the surface integral of \eqref{J3.5}, we have
\begin{align*}
|\int_{=} (\nabla K |\nabla K|^2)|u|^2|\leq C\int_{=} |u|^2 \leq C |||u|||_1^2.
\end{align*}

After all of this, since $R$ and $R_1$ are comparable, we have shown the proposition.

\end{proof}

As proven in \cite{perthame_morreycampanato_1999}, the estimate
\begin{equation*}
\lambda |||u|||_1^2 + |||\nabla u|||_1^2 \leq C(1+\epsilon)(N_1(f))^2
\end{equation*}
holds under our assumptions, and we have the corollary
\begin{corollary}
Under the same assumptions as Proposition \ref{prop1}, we have
\begin{align*}
&\int_{K\leq R} |\nabla K|^2 K \left|\nabla_K^r u-i\lambda^{1/2}|\nabla K| u+\frac{d-1}{2K}|\nabla K|u\right|^2 + R\int_{K\geq 2R} |\nabla K|^2|\nabla (e^{-i\lambda^{1/2}K}u)|^2 \\
&+\frac{\epsilon}{2\lambda^{1/2}}\int_{K\leq R} K^2 \left| \nabla u-i\lambda^{1/2} \nabla K u\right|^2
+\frac{\epsilon R}{2\lambda^{1/2}}\int_{K \geq 2R} K \left| \nabla u-i\lambda^{1/2} \nabla K u\right|^2\\
\leq &C\left[ (1+\epsilon)(N_1(f))^2 + \int (1+\epsilon |x|) |x|^3 |f|^2 \right]
\end{align*}
\end{corollary}
Using the same arguments as Zubeldia \cite{zubeldia_limiting_2013}, if $u$ is given by the limited absorption  principle \eqref{res_sol}, then $u$ is a solution in $H^1_{loc}(\mathbb{R}^d)$ of 
\[
\Delta u + (Q+p)u + \lambda u = f,
\]
which satisfies
\begin{align*}
&\int_{K\leq R} |\nabla K|^2 K \left|\nabla_K^r u-i\lambda^{1/2}|\nabla K| u+\frac{d-1}{2K}|\nabla K|u\right|^2 + R\int_{K\geq 2R} |\nabla K|^2|\nabla (e^{-i\lambda^{1/2}K}u)|^2 \\
\leq &C\left[ (N_1(f))^2 + \int  |x|^3 |f|^2 \right],
\end{align*}
where $C$ is independent of $R$. By taking the supremum over all $R$, we have finally proven our main theorem.

\section{Outlook}
In their paper, Barcelo, Vega and Zubeldia consider the \textit{electromagnetic} Helmholtz equation 
\begin{equation}
\label{electromag}
\nabla_A^2 u + Vu + \lambda u =f.
\end{equation}
with $\nabla_A = \nabla + iA$, and $A$ a magnetic vector potential. While we have chosen to set $A=0$, we expect that Theorem \ref{Theorem_to_prove} remains true if we replace $\nabla$ with $\nabla_A$, when $A$ obeys the conditions for Theorem 1.7 of \cite{barcelo_forward_2012}. Whether the non-spherical method in this paper can be used to improve these conditions on $A$ is less clear, but it would certainly be an interesting topic of further research. 

While our method allows for the consideration of potentials with less decay towards infinity, it does not improve the result of \cite{barcelo_forward_2012} for singularities around the origin. Adapting the method of this paper could be a way to prove that Theorem \ref{Theorem_to_prove} can be modified to allow for potentials behaving like $|x|^{-\gamma}$ close to $0$, for $\gamma$ between two and three. We propose this as an interesting question to be investigated elsewhere.

\bibliographystyle{chapters/CUP}  

\end{document}